\documentclass[a4paper]{amsart}
\usepackage{amsfonts,amsthm,amsmath,amssymb}
\usepackage{graphicx}
\usepackage{epsfig}
\usepackage{tikz}
\usepackage{youngtab}
\usepackage{todonotes}
\usepackage{hyperref}

\theoremstyle{plain}
\newtheorem{theorem}{Theorem}[section]
\newtheorem{prop}[theorem]{Proposition}
\newtheorem{lemma}[theorem]{Lemma}
\newtheorem{corollary}[theorem]{Corollary}
\newtheorem{conjecture}[theorem]{Conjecture}
\newtheorem{question}[theorem]{Question}

\theoremstyle{remark}

\newtheorem{remark}[theorem]{Remark}

\theoremstyle{definition}
\newtheorem{definition}[theorem]{Definition}
\newtheorem{example}[theorem]{Example}

\usepackage{amsmath,amssymb}

\newcommand{\ZZ}{{\mathbb{Z}}}

\newcommand{\cT}{{\mathcal{T}}}

\newcommand{\fS}{{\mathfrak{S}}}

\newcommand{\GL}{{\mathop{\mathrm{GL}}}}

\newcommand{\Cat}{\mathop{\mathrm{Cat}}}

\newcommand{\wt}{\mathop{\mathrm{wt}}}

\title[Determinantal identities]{Determinantal identities for flagged Schur and Schubert polynomials}

\author{Grigory Merzon}
\address{Moscow Center for Continuous Mathematical Education, Bolshoi Vlassievskii per., 11, 119002 Moscow, Russia }
\email{merzon@mccme.ru}
\author{Evgeny Smirnov}
\address{Department of Mathematics \& Laboratory of Algebraic Geometry and its Applications, National Research University Higher School of Economics, Vavilova st., 7, 119312 Moscow, Russia}
\address{Independent University of Moscow, Bolshoi Vlassievskii per., 11, 119002 Moscow, Russia }
\email{esmirnov@hse.ru}

\begin{document}

\begin{abstract}
We prove new determinantal identities for a family of flagged Schur polynomials. As a corollary of these identities we obtain determinantal expressions of Schubert polynomials for certain vexillary permutations. 
\end{abstract}

\maketitle

\section{Introduction}

Schur polynomials  were defined by Issai Schur \cite{Schur1901} as characters of irreducible representations of $\GL(n)$; they are symmetric polynomials in $n$ variables, indexed by Young diagrams $\lambda$ with at most $n$ rows. They form a basis in the ring of all symmetric polynomials. This basis plays a prominent role in various algebraic problems.

Schur polynomials also have a  combinatorial definition: they are obtained as sums of monomials indexed by semistandard Young tableaux of shape $\lambda$. These are fillings of boxes of $\lambda$ by nonnegative integers satisfying certain conditions. This construction is due to Specht \cite{Specht35}, who showed that such tableaux index a basis in the representation of $\GL(n)$ corresponding to $\lambda$.

The classical Jacobi--Trudi formula provides a determinantal expression for Schur polynomials via \emph{complete symmetric functions}: for a partition $\lambda=(\lambda_1,\lambda_2,\dots,\lambda_m)$, the corresponding Schur polynomial $s_\lambda(x_1,\dots,x_n)$ is equal to
\[
s_\lambda(x_1,\dots,x_n)=\det\left( h_{\lambda_i-i+j}\right)_{i,j=1}^m,
\]
where 
\[
h_k=\sum_{i_1\leq \dots\leq i_k} x_{i_1}\dots x_{i_k}
\]
is the $k$-th complete symmetric function, i.e., the sum of all monomials of degree $k$ in variables $x_1,\dots,x_n$.

In 1982, A.~Lascoux and M.-P.~Sch\"utzenberger generalized the notion of Schur polynomials, defining \emph{flagged Schur polynomials}. Their definition is similar to the combinatorial definition of Schur polynomials, with the only difference: they put extra constraints on the entries occuring in semistandard Young tableaux, encoded by a sequence of integers referred to as a \emph{flag}. These polynomials are not symmetric anymore. However, they also satisfy an analogue of the Jacobi--Trudi identity, due to I.~Gessel and M.~Wachs. 

Flagged Schur polynomials were defined because of their relation to \emph{Schubert polynomials}. This is another remarkable family of polynomials. It is also due to Lascoux and Sch\"utzenberger. These polynomials are indexed by permutations; they represent cohomology classes of Schubert cycles in flag varieties. Just as Schur polynomials, they admit an algebraic definition (we recall it in Subsection~\ref{ssec:defschubert}), as well as a combinatorial one, due to S.~Fomin and An.~Kirillov~\cite{FominKirillov96}. The latter definition uses certain combinatorial objects called \emph{pipe dreams}, or \emph{rc-graphs}. It turns out that for a ``nice'' family of permutations, known as \emph{vexillary permutations}, Schubert polynomials are equal to some flagged Schur polynomials. This gives determinantal expressions for Schubert polynomials of vexillary permutations.

In this paper we consider flagged Schur polynomials with flags of some special type (the so-called $h$-flagged ones). We show that for these polynomials there is a formula expressing them as determinants whose entries are 1-flagged Schur polynomials, divided by certain monomial denominators. This determinantal formula is different from the generalized Jacobi--Trudi formula discussed above.

 The proof of our formula is based on the well-known Lindstr\"om--Gessel--Viennot lemma. We present flagged Schur polynomials as weighted sums over tuples of nonintersecting lattice paths and express these sums as determinants using this lemma.

Then we reinterpret our formula in terms of Schubert polynomials. It expresses the Schubert polynomial of an $h$-shifted dominant permutation as a determinant whose entries are Schubert polynomials of 1-shifted dominant permutations. In particular, this gives determinantal formulas for the Schubert polynomials corresponding to the conjecturally ``most singular'' Schubert varieties.

\subsection*{Structure of the paper} This paper is organized as follows. In Section~\ref{sec:flagged} we give the definition of flagged Schur varieties and recall the generalized Jacobi--Trudi identity. In Section~\ref{sec:another} we present a path interpretation of $h$-flagged Schur polynomials and prove our main result, Theorem~\ref{thm:flagged}. To make our exposition complete, in Subsection~\ref{ssec:LGV} we give  a proof of our main combinatorial tool, the Lindstr\"om--Gessel--Viennot lemma. Section~\ref{sec:schubert} is devoted to Schubert polynomials. We recall their definition, discuss their relation with flagged Schur polynomials and reformulate our main result in terms of Schubert polynomials (Theorem~\ref{thm:mainschubert}). Finally, in Section~\ref{sec:speculations} we speak about its applications to some special classes of permutations and discuss its (mostly conjectural) relations to the geometry of Schubert varieties.

\section{Flagged Schur polynomials and the Jacobi--Trudi identity}\label{sec:flagged}

\subsection{Definition of flagged Schur polynomials} We start with recalling the definition of flagged Schur polynomials, introduced by A.~Lascoux and M.-P.~Sch\"utzenberger \cite{LascouxSchutzenberger82}. 

Fix two positive integers $m\leq n$. Let $\lambda=(\lambda_1\geq \dots\geq\lambda_m>0)$ be a partition and let $b=(b_1\leq b_2\leq\dots\leq b_m=n)$ a sequence of nonstrictly increasing positive integers. A \emph{flagged semistandard Young tableau} $T$ of shape $\lambda$ and flags $b$ is an array $t_{ij}$ of positive integers $t_{ij}$, $1\leq j\leq\lambda_i$, $1\leq i\leq m$, such that $1\leq t_{ij}\leq t_{i,j+1}\leq b_i$ and $t_{ij}<t_{i+1,j}$. In other words, the boxes of a Young diagram of shape $\lambda$ are filled by numbers $t_{ij}$ non-strictly increasing along the rows and strictly increasing along the columns, in such a way that the numbers in the $i$-th row do not exceed $b_i$. Denote the set of all such tableaux by $\cT(\lambda,b)$. To each tableau $T$ we assign a monomial $M(T)=x_1^{p_1}\dots x_n^{p_n}$, where $p_k$ is the number of entries in $T$ equal to $k$.

\begin{definition}
A \emph{flagged Schur polynomial} of shape $\lambda$ and flags $b$ is defined as 
\[
 s_\lambda(b)=\sum_{T\in \cT(\lambda,b)} M(T)\in \ZZ[x_1,\dots,x_n].
\]
\end{definition}


\begin{example} If $b_1=\dots=b_m=n$, then $\cT(\lambda,b)$ is the set of semistandard Young tableaux of shape $\lambda$ whose entries do not exceed $n$. In this case $s_\lambda(b)$ is just the usual Schur polynomial $s_\lambda(x_1,\dots,x_n)$ in $n$ variables.
\end{example}

\subsection{$h$-flagged Schur polynomials} We will be mostly interested in flagged Schur polynomials with flags of a special form.

\begin{definition} Let $h$ be a positive integer, $b=(h+1,h+2,\dots, h+m)$. Then 
$s_{\lambda}^{(h)}=s_\lambda(b)$ is called an \emph{$h$-flagged Schur polynomial of shape $\lambda$}. We will also speak about \emph{$h$-flagged semistandard Young tableaux of shape $\lambda$}.
\end{definition}

\begin{example}\label{1-flagged}
$1$-flagged semistandard Young tableaux of shape $\lambda$ correspond to Young subdiagrams $\mu\subset\lambda$. Indeed, let $T$ be such a tableau. The entries $t_{ij}$ in each ($i$-th) row of $T$ are equal either to $i$ or to $i+1$. The union of all boxes $(i,j)\in T$ such that $t_{ij}=i$ forms a subdiagram $\mu$, and $\mu$ uniquely determines the tableau $T$.
\end{example}

\begin{example} $h$-flagged semistandard Young tableaux of shape $\lambda$ correspond to $h$-tuples 
\[
\mu_1\subseteq\mu_2\subseteq\dots\subseteq\mu_h\subseteq\lambda
\] 
of embedded Young diagrams, where $\mu_\ell=\{(i,j)\mid t_{ij}\leq i-1+\ell\}$. Alternatively, replacing each entry $t_{ij}$ by $t_{ij}'=h+i-t_{ij}$ transforms a semistandard Young tableau $T$ into a \emph{plane partition} $T'$, i.e. a tableau of the same shape whose nonnegative entries \emph{non-strictly decrease} along both rows and columns. Such tableau can be viewed as a three-dimensional Young diagram of height at most $h$ whose base fits in $\lambda$.
\end{example}

\begin{example} $s_{(2,1)}^{(1)}(x_1,x_2,x_3)=x_1^2x_2+x_1^2x_3+x_1x_2x_3+x_1x^2_2+x_2^2x_3$. The set of 1-flagged Young tableaux for the diagram $(2,1)$ is as follows:
\[
\young(11,2)\qquad\young(11,3)\qquad\young(12,3)\qquad\young(12,2)
\qquad \young(22,3)
\]
\end{example}


\subsection{Generalized Jacobi--Trudi identity} The following determinantal expression for flagged Schubert polynomials is well known. It appeared first in the paper \cite{Gessel82} by  I.~Gessel. Its proof is also given in M.~Wachs's paper \cite[Theorem~1.3]{Wachs85}. In their paper~\cite{ChenLiLouck02}, Chen, Li and Louck generalize this identity for the case of double flagged Schur  polynomials and give yet another its proof, using the lattice path interpretation of these polynomials and the Lindstr\"om--Gessel--Viennot lemma.

\begin{theorem}[\cite{Gessel82}, \cite{Wachs85}, \cite{ChenLiLouck02}]
\label{genJT} Let $s_\lambda(b)$ be a flagged Schur function of shape $\lambda$ and flags $b$. Then
\[
 s_\lambda(b)=\det(h_{\lambda_i-i+j}(b_i))_{1\leq i,j\leq m},
\]
where 
\[
h_d(k)= \sum_{1\leq i_1\leq\dots\leq i_d\leq k} x_{i_1}x_{i_2}\dots x_{i_d}
\]
is the complete symmetric function of degree $d$ in $k$ variables, $h_0=1$, and $h_\ell=0$ for $\ell<0$.
\end{theorem}

\begin{corollary} For $b=(n,\dots,n)$, this is the classical Jacobi--Trudi formula for Schur polynomials in $n$ variables:
\[
 s_\lambda(x_1,\dots,x_n)=\det(h_{\lambda_i-i+j}(x_1,\dots,x_n))_{1\leq i,j\leq m}.
\]
\end{corollary}

\begin{remark}
The ``path interpretation'' of the Jacobi--Trudi formula and its proof using the Lindstr\"om--Gessel--Viennot lemma can be found, for instance, in \cite[7.16]{Stanley2}.
\end{remark}

\section{Another determinantal formula for flagged Schur polynomials}\label{sec:another}

\subsection{Path interpretation of 1-flagged Schur polynomials} Let $\lambda=(\lambda_1,\dots,\lambda_{n})$ be a fixed partition. Consider an oriented graph $G_\lambda$ whose vertices are lattice points inside $\lambda$ or at its boundary, and edges are given by vertical and horizontal line segments joining the neighboring lattice points. The orientation on the edges is as follows: the horizontal and the vertical edges are directed towards east and north, respectively.

Let us assign weights to the edges of $G_\lambda$ in the following way: to each horizontal edge $(r,-s)-(r+1,-s)$ we assign the weight $x^{-1}_{s+1}$, as shown on Fig.~\ref{fig:weights}, while all vertical edges have weight 1. For a path $P$, its weight $\wt P$ is defined as the product of the weights of all its edges.

\begin{definition}
Let $A$ and $B$ be two vertices of $G_\lambda$. The \emph{partition function} $Z_\lambda(A,B)$ is defined as the weighted sum over all paths from $A$ to $B$:
\[
Z_\lambda(A,B)=\sum_{P\colon A\to B}\wt P.
\]
\end{definition}

The following easy proposition expresses 1-flagged Schur polynomials via partition functions

\begin{prop}\label{prop:oneflag}
\[
s_{\lambda}^{(1)}(x_1,\dots,x_{n+1}) =x_1^{\lambda_1}x_2^{\lambda_1}\dots x_{n+1}^{\lambda_n}Z_\lambda(A,B).
\]
\end{prop}

\begin{proof} As it was discussed in Example~\ref{1-flagged}, the Young subdiagrams $\mu\subset \lambda$ correspond to 1-flagged semistandard Young tableaux of the shape $\lambda$. Each of these tableaux defines a monomial in $x_1,\dots,x_{n+1}$, which we denote by $x^\mu$. On the other hand, let  $P_\mu$ be the path leading from $A$ to $B$ and bounding the diagram $\mu$ from below. Let us prove that 
\[
x^\mu=x_1^{\lambda_1}x_2^{\lambda_1}\dots x_{n+1}^{\lambda_n}\wt P_\mu.
\] 
And this can be proven by induction on the number of boxes in $\mu$: for $\mu=\varnothing$, the weight of the corresponding tableau is equal to $x_2^{\lambda_1}x_3^{\lambda_2}\dots x_{n+1}^{\lambda_n}=x_1^{\lambda_1}x_2^{\lambda_1}\dots x_{n+1}^{\lambda_n}/x_1^{\lambda_1}$. Adding a box in the $i$-th row of $\mu$ results in multiplying this weight by $x_i/x_{i+1}$.
\end{proof}

Let us introduce the following notation. Let $\mu=(\mu_1,\dots,\mu_m)$ be a Young diagram, and $(y_1,\dots,y_m)$ an (ordered) sequence of variables, with the number of variables equal to the number of rows in $\mu$. We set
\[
\underline{(y_1,\dots,y_m)}^\mu:=y_1^{\mu_1}\dots y_m^{\mu_m}.
\]
With this notation, Proposition~\ref{prop:oneflag} can be rewritten as follows:
\[
s_\lambda(x_1,\dots,x_{m+1})=x_1^{\lambda_1}\underline{(x_2,\dots,x_{n+1})}^\lambda \cdot Z_\lambda(A,B).
\]

\begin{figure}[h!]
\begin{tikzpicture}[nodes={font=\scriptsize}]
\draw (0,0) grid (5,4);
\foreach \y in {5,...,1}
	\foreach \x in {0,...,4}
		\node at (\x+.5,5-\y+.2) {$x^{-1}_{\y}$};
\fill[white] (2,-0.5)--(2,2)--(3,2)--(3,3)--(5.5,3)--(5.5,-0.5)--cycle;
\draw (0,0)--(2,0)--(2,2)--(3,2)--(3,3)--(5,3)--(5,4);

\begin{scope}[ultra thick, rounded corners]
\draw[red] 
(0,0)--(0,1)--(2,1)--(2,3)--(3,3)--(3,4)--(5,4);
\end{scope}
\draw[fill] (0,0) circle [radius=0.1];
\node[below right] at (0,0) {$A$};
\draw[fill] (5,4) circle [radius=0.1];
\node[below right] at (5,4) {$B$};

\end{tikzpicture}
\caption{Correspondence between subdiagrams and paths}\label{fig:weights}
\end{figure}
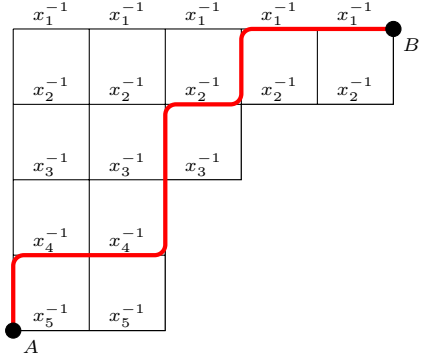

\begin{example} The diagram $\lambda=(5,3,2,2)$ with the weights on its edges is shown on  Figure~\ref{fig:weights}. Consider the path $P_\mu$ drawn in red; its weight equals $\wt P_\mu=x_1^{-2}x_2^{-1}x_4^{-2}$. The corresponding summand in the flagged Schur polynomial is then equal to
\[
 x_1^5x_2^5x_3^3x_4^2x_5^2\cdot\wt P_\mu=x_1^3x_2^4x_3^3x_5^2.
\]
The Young tableau corresponding to $\mu$ is as follows:
\[
 \young(11122,223,33,55)
\]
\end{example}

\subsection{Expressing $h$-flagged Schur polynomials via $1$-flagged Schur polynomials} 

In this subsection we give a lattice path expression for $h$-flagged Schur polynomials, different from the one given in Theorem~\ref{genJT}. Using the Lindstr\"om--Gessel--Viennot lemma, we express them as determinants involving 1-flagged polynomials from the previous subsection.

As before, let $\lambda=(\lambda_1,\dots,\lambda_m)$ be a Young tableau, and let $h\geq 1$. We consider $h$-flagged Young tableaux: the tableaux with flags $(h+1,\dots,h+m)$. To each such tableau $T$ we can assign an $h$-tuple of lattice paths drawn in the fourth quarter of the plane that start at the point $A=(0,-m)$, end at $B=(\ell,0)$, where $\ell=\lambda_1$, and pass inside the diagram $\lambda$. These paths are constructed in the following way. For $k\leq h$, denote by $\mu_k$ the set of boxes $t_{ij}$ of $T$ such that $t_{ij}\leq i+k-1$. We thus obtain a set of embedded Young tableaux $\mu_1\subset\mu_2\subset\dots\subset \mu_h\subset\lambda$. Consider the set of lattice paths $P_1,\dots,P_h$ bounding these diagrams; they pass non-strictly above one another, and all of them are bounded by $\lambda$ from below.

\begin{example} Figure~\ref{fig:twopaths} below shows a 2-flagged Young tableau for the diagram $\lambda=(5,4,4,2)$.
It corresponds to the pair of the blue and the red paths, shown at the figure on the right.
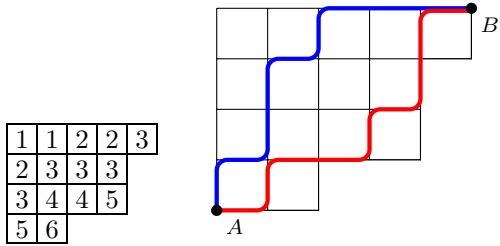
\begin{figure}[h!]
 $\young(11223,2333,3445,56)$\qquad 
\begin{tikzpicture}[font=\footnotesize,scale=0.67]
\draw (0,0) grid (5,4);
\fill [white](2.01,-0.05)--(5.5,-0.05)--(5.5,2.99)--(4.01,2.99)--(4.01,0.99)--(2.01,0.99)--cycle;
\draw (0,0)--(2,0)--(2,1)--(4,1)--(4,3)--(5,3)--(5,4);
\begin{scope}[ultra thick, rounded corners]
\draw[blue] 
(0,0)--(0,1)--(1,1)--(1,3)--(2,3)--(2,4)--(5,4);
\draw[red]
(0,0)--(1,0)--(1,1)--(3,1)--(3,2)--(4,2)--(4,3.95)--(5,3.95);
\end{scope}
\draw[fill] (0,0) circle [radius=0.1];
\node[below right] at (0,0) {$A$};
\draw[fill] (5,4) circle [radius=0.1];
\node[below right] at (5,4) {$B$};
\end{tikzpicture}
\caption{Paths corresponding to a 2-flagged tableau}\label{fig:twopaths}
\end{figure}
\end{example}

Now let us shift each ($k$-th) path to the vector $(k-1,-k+1)$. This will give us an $h$-tuple of \emph{non-crossing} paths. The $k$-th path $P_k$ will then start at the point $A_k=(k-1,-m-k+1)$ and end at $B_k=(\ell+k-1,-k+1)$.

Introduce weights at the horizontal segments of the lattice as in the previous example. 
One can easily see that the monomial corresponding to $T$ is nothing but
\[
 x^T=\left[(x_1\dots x_{h})^{\lambda_1}\cdot\underline{(x_{h+1},\dots,x_{h+m})}^\lambda\right]\cdot\wt P_1\dots\wt P_h.
\]
Note that the multiple in front of the product of the weights is the inverse to the product of the weights of \emph{all} edges in this diagram.

\begin{figure}[h!]
\begin{tikzpicture}[nodes={font=\scriptsize}]
\fill[black!15!white] (0,0)--(1,0)--(1,4)--(6,4)--(6,5)--(0,5)--cycle;
\draw (0,0) grid (6,5);
\foreach \y in {6,...,1}
	\foreach \x in {0,...,5}
		\node at (\x+.5,6-\y+.2) {$x^{-1}_{\y}$};
\fill [white] (3,-0.5)--(6.5,-0.5)--(6.5,3)--(5,3)--(5,1)--(3,1)--cycle;
\draw (6,5)--(6,3)--(5,3)--(5,1)--(3,1)--(3,0)--(0,0);
\begin{scope}[ultra thick, rounded corners]
\draw[blue] 
(0,1)--(0,2)--(1,2)--(1,4)--(2,4)--(2,5)--(5,5);
\draw[red]
(1,0)--(2,0)--(2,1)--(4,1)--(4,2)--(5,2)--(5,4)--(6,4);
\draw[blue, dashed] (0,0)--(0,1);
\draw[blue,dashed] (5,5)--(6,5);
\end{scope}

\draw[fill] (1,0) circle [radius=0.1];
\node[below right] at (1,0) {$A_2$};
\draw[fill] (6,4) circle [radius=0.1];
\node[below right] at (6,4) {$B_2$};
\draw[fill] (0,1) circle [radius=0.1];
\node[below right] at (0,1) {$A_1$};
\draw[fill] (5,5) circle [radius=0.1];
\node[below right] at (5,5) {$B_1$};
\draw[fill] (0,0) circle [radius=0.1];
\node[below right] at (0,0) {$A_1'$};
\draw[fill] (6,5) circle [radius=0.1];
\node[below right] at (6,5) {$B_1'$};
\end{tikzpicture}
\caption{Shifted and extended paths}\label{fig:shifted}
\end{figure}

Let us also extend paths as shown on the picture: the $k$-th path $P_k$ is extended by $h-k$ vertical segments down and by $h-k$ horizontal segments to the right. Denote the extended path by $P_k'$ and its endpoints by $A_k'$ and $B_k'$. The extensions are shown on Figure~\ref{fig:shifted} by dashed lines. 

It is clear that the weights of $P_k$ and $P_k'$ differ by a ``staircase monomial'': $\wt P_k=\wt P_k'\cdot x_k^{h-k}$. This means that
\[
 x^T=\left[(x_1^{\lambda_1+h-1}\dots x_{h-1}^{\lambda_1+1} x_h^{\lambda_1})\underline{(x_{h+1},\dots,x_{h+m})}^\lambda\right]\cdot\wt P_1'\dots\wt P_h'.
\]

This allows us to express $h$-flagged Schur polynomials as weighted  sums over sets of \emph{noncrossing} paths. To do this, let us introduce some further notation.

\begin{definition}\label{def:extended} Let $\lambda=(\lambda_1,\dots,\lambda_m)$ be a Young diagram, and $k,\ell$ two nonnegative integers. Let us shift $\lambda$ by the vector $(\ell,-k)$. The minimal Young diagram containing this set is denoted by $\lambda(k,\ell)$. In other terms,  $\lambda[k,\ell]$  is obtained from $\lambda$ by adding $k$ rows of length $\lambda_1$ from above and $\ell$ columns of height $m$ from the left.
\end{definition}

Consider the diagram $\lambda(h-1,h-1)$. Let $A_i'$ and $B_i'$ be as shown in Figure~\ref{fig:shifted}: namely, $A'_i=(i-1,-m-h+1)$ and  $B'_i=( \lambda_1+h-1,1-i)$. Denote by $Z_{nc}((A_1',B_1'),\dots,(A_h',B_h'))$ the weighted sum taken over all $h$-tuples of \emph{non-crossing} paths joining $A'_1$ with $B_1'$, \dots $A_h'$ with $B'_h$ and contained in the diagram $\lambda[h-1,h-1]$. (The index ``nc'' stands for ``non-crossing''). The above discussion shows that the $h$-flagged Schur function $s_\lambda^{h}(x_1,\dots,x_n)$ equals
\begin{multline*}
s_\lambda^{(h)}(x_1,\dots,x_{m+h})=\left[(x_1^{\lambda_1+h-1}\dots x_{h-1}^{\lambda_1+1} x_h^{\lambda_1})
\underline{( x_{h+1},\dots,x_{h+m})}^{\lambda}\right]
\times\\ 
\times Z_{nc}((A'_1,B'_1),\dots,(A'_h,B'_h)).
\end{multline*}

The following statement is a particular case of the well-known Lindstr\"om--Gessel--Viennot lemma. To make our exposition complete, we formulate this lemma in full generality and give its proof in Subsection~\ref{ssec:LGV}.

\begin{lemma} The following identity holds:
\[
Z_{nc}((A'_1,B'_1),\dots,(A'_h,B'_h))= \det\left( Z(A'_i,B'_j) \right)_{i,j=1}^h.
\]
\end{lemma}

This lemma together with the expressions for Schur polynomials via lattice paths gives us the following result which expresses an $h$-flagged Schur polynomial via 1-flagged Schur polynomials.

\begin{theorem}\label{thm:flagged} The following equality holds:
\begin{multline*}
s_\lambda^{(h)}(x_1,\dots,x_{h+m})=\left[(x_1^{\lambda_1+h-1}\dots x_{h-1}^{\lambda_1+1} x_h^{\lambda_1})
\underline{(x_{h+1},\dots, x_{h+m})}^{\lambda}\right]\times\\
\times\det\left(\frac{s^{(1)}_{\lambda[i-1,j-1]}(x_j,x_{j+1},\dots,x_{h+m})}{x_j^{\lambda_1+h-i} \underline{(x_{j+1},\dots,x_{h+m})}^{\lambda[i-1,j-1]}}\right)_{i,j=1}^h
\end{multline*}
\end{theorem}

\begin{remark}
This determinantal formula is different from the Jacobi--Trudi formula (and, to the best of our knowledge, cannot be reduced to it). Note, in particular, that all matrix entries in the Jacobi--Trudi formula are polynomials in $x_1,\dots,x_n$, while the matrix elements in our formula  are polynomials in \emph{different} sets of variables. 
\end{remark}

\begin{example} Let $h=2$, $\lambda=(2,1)$. Then our formula looks as follows:
\[
\frac{s^{(2)}_{(2,1)}(x_1,x_2,x_3,x_4)}{x_1^3x_2^2x_3^2x_4}=\det\begin{pmatrix}
\frac{s_{(2,1)}^{(1)}(x_2,x_3,x_4)}{x_2^2x_3^2x_4} & 
\frac{s_{(2,2,1)}^{(1)}(x_1,x_2,x_3,x_4)}{x_1^2x_2^2x_3^2x_4}\\
\frac{s_{(3,2)}^{(1)}(x_2,x_3,x_4)}{x_2^3x_3^3x_4^2} & 
\frac{s_{(3,3,2)}^{(1)}(x_1,x_2,x_3,x_4)}{x_1^3x_2^3x_3^3x_4^2}
\end{pmatrix}.
\]

\end{example}

We will also need a slightly different form of this determinantal formula. It can be obtained from Theorem~\ref{thm:flagged} by algebraic manipulations with determinants, but we will give a combinatorial proof instead.

Let us modify Definition~\ref{def:extended} in the following way. 

\begin{definition}\label{extended2} Let $\lambda=(\lambda_1,\dots,\lambda_m)$ be a Young diagram. A \emph{$(k,\ell)$-extended} diagram (notation: $\widehat\lambda[k,\ell]$) is defined as follows:
\[
\widehat{\lambda}[k,\ell]=(\lambda_1+k+\ell,\lambda_1+k-1+\ell,\dots,\lambda_1+\ell,\lambda_2+\ell,\dots,\lambda_m+\ell,\ell,\ell-1,\dots,1).
\]
\end{definition}
In other words, $\widehat\lambda[k,\ell]$ is obtained from $\lambda$ by adding ``staircases'' of size $k$ and $\ell$ on the top and to the left of it, respectively.

\begin{example} This picture shows $\lambda=(3,3,1)$ and $\widehat\lambda[2,3]$.
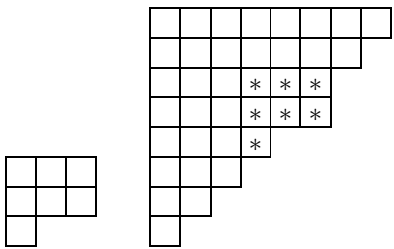
\[
\yng(3,3,1) \qquad \young(~~~~~~~~,~~~~~~~,~~~***,~~~***,~~~*,~~~,~~,~)
\]
\end{example}

Now look at Figure~\ref{fig:shifted} again. We can extend the lengths of paths not by 1, 2, \dots, $h-1$, but by 2, 4, \dots, $2(h-1)$. The result is shown on Figure~\ref{fig:another_shifted}.

\begin{figure}[h!]
\begin{tikzpicture}[nodes={font=\scriptsize}]
\fill[black!15!white]
(0,0)--(0,6)--(7,6)--(7,5)--(1,5)--(1,0)--cycle;
\draw (0,0) grid (7,6);
\foreach \y in {7,...,1}
	\foreach \x in {0,...,6}
		\node at (\x+.5,7-\y+.2) {$x^{-1}_{\y}$};
\fill [white] (1,-0.5)--(7.5,0-.5)--(7.5,5)--(6,5)--(6,4)--(5,4)--(5,2)--(3,2)--(3,1)--(1,1)--cycle;
\draw (7,5)--(6,5)--(6,4)--(5,4)--(5,2)--(3,2)--(3,1)--(1,1)--(1,0);
\begin{scope}[ultra thick, rounded corners]
\draw[blue] 
(0,2)--(0,3)--(1,3)--(1,5)--(2,5)--(2,6)--(5,6);
\draw[red]
(1,1)--(2,1)--(2,2)--(4,2)--(4,3)--(5,3)--(5,5)--(6,5);
\end{scope}
\begin{scope}[ultra thick, dashed]
\draw[blue] 
(0,2)--(0,0);
\draw[blue](5,6)--(7,6);
\end{scope}

\draw[fill] (1,1) circle [radius=0.1];
\node[below right] at (1,1) {$A_2$};
\draw[fill] (6,5) circle [radius=0.1];
\node[below right] at (6,5) {$B_2$};
\draw[fill] (0,2) circle [radius=0.1];
\node[below right] at (0,2) {$A_1$};
\draw[fill] (5,6) circle [radius=0.1];
\node[below right] at (5,6) {$B_1$};
\draw[fill] (0,0) circle [radius=0.1];
\node[below right] at (0,0) {$A_1''$};
\draw[fill] (7,6) circle [radius=0.1];
\node[below right] at (7,6) {$B_1''$};
\end{tikzpicture}
\caption{Another way of extending paths}\label{fig:another_shifted}
\end{figure}
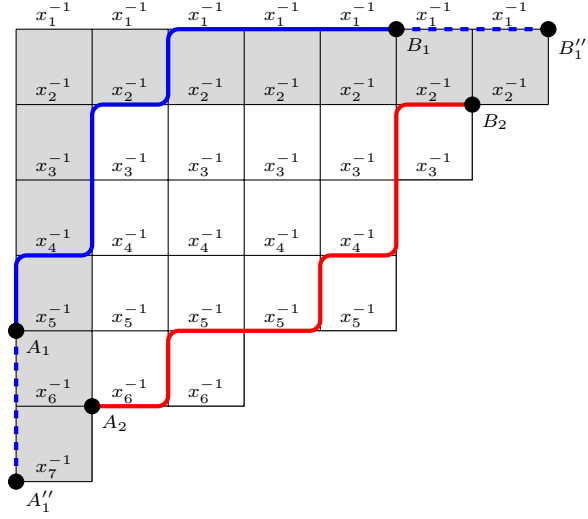

Note that the $k$-th path is situated inside the diagram $\widehat\lambda[h-k,h-k]$.

This allows us to formulate the following corollary of Theorem~\ref{thm:flagged}. Its proof also follows from the Lindstr\"om--Gessel--Viennot lemma; the only difference concerns the weights which are added while extending paths.

\begin{corollary}\label{thm:flagged2} The following equality holds:
\begin{multline*}
s_\lambda^{(h)}(x_1,\dots,x_{h+m})=\left[(x_1^{\lambda_1+2(h-1)}x_2^{\lambda_1+2(h-2)}\dots x_{h-1}^{\lambda_1+2} x_h^{\lambda_1})
\underline{(x_{h+1},\dots, x_{h+m})}^{\lambda}\right]\times\\
\times\det\left(\frac{s^{(1)}_{\widehat\lambda[i-1,j-1]}(x_j,x_{j+1},\dots,x_{2h+m-i})}{x_j^{\lambda_1+2h+2-i-j} \underline{(x_{j+1},\dots,x_{h+m})}^{\widehat\lambda[i-1,j-1]}}\right)_{i,j=1}^h
\end{multline*}
\end{corollary}

\begin{remark} This corollary can also be obtained from Theorem~\ref{thm:flagged} by purely algebraic manipulations with the determinant.
\end{remark}

\begin{remark} Note that the right-hand side depends upon $x_k$ for $k>h+m$, while the left-hand side does not. This means that all these $x_k$'s cancel out while evaluating the determinant.
\end{remark}

We will need this form of our determinantal identity while dealing with Schubert polynomials in the next section.

\subsection{The Lindstr\"om--Gessel--Viennot lemma}\label{ssec:LGV}
In this subsection we prove the Lindstr\"om--Gessel--Viennot lemma on noncrossing paths. This material is by no means new: having appeared first in the preprint \cite{GesselViennot86}, now it can be found in many books and articles, for instance, in \cite{Stanley1}.

Let $\Gamma$ be an arbitrary oriented graph without oriented cycles, with weights assigned to its edges (we suppose that the weights are elements of a certain ring $R$). For a path $P$ joining two vertices $A$ and $B$, by the weight $\wt P$ of $P$ we mean the product of weights of all of its edges. As before, we define the \emph{partition function} $Z(A,B)$ to be the sum of weights of all paths joining $A$ with $B$:
\[
Z(A,B)=\sum_{P\colon A\to B}\wt P.
\]
An $h$-tuple of paths $P_1\colon A_1\to B_1,\dots P_h\colon A_h\to B_h$ is said to be \emph{noncrossing} if any two paths from it have no common points. An (ordered) set of $h$ starting points and $h$ endpoints $(A_1,\dots,A_h)$ and $(B_1,\dots,B_h)$ is said to be \emph{compatible} if for any permutation $\sigma\in S_h$ the existence of a noncrossing $h$-tuple of paths joining $A_1$ with $B_{\sigma(1)}$,\dots, $A_h$ with $B_{\sigma(h)}$ implies that $\sigma$ is the identity permutation.

For two $h$-tuples of starting points and endpoints, consider a weighted sum over all $h$-tuples of noncrossing paths $(P_1,\dots,P_h)$:
\[
Z_{nc}((A_1,B_1),\dots,(A_h,B_h))=\sum_{P_i\colon A_i\to B_i,\quad (P_1,\dots,P_h)\text{ noncrossing}} \wt P_1\cdot\dots\cdot\wt P_h.
\]

\begin{lemma}[Lindstr\"om--Gessel--Viennot] If $(A_1,\dots,A_h)$ and $(B_1,\dots,B_h)$ are compatible, the function $Z_{nc}((A_1,B_1),\dots,(A_h,B_h))$ admits the following determinantal expression:
\[
Z_{nc}((A_1,B_1),\dots,(A_h,B_h))=\det\left(Z(A_i,B_j)\right)_{i,j=1}^h.
\]
\end{lemma}

\begin{proof}
The determinant in the right-hand side can be expressed as the sum over permutations:
\[
\det\left(Z(A_i,B_j)\right)_{i,j=1}^h=\sum_{\sigma\in S_h}(-1)^\sigma Z(A_1,B_{\sigma(1)})\cdot\dots\cdot Z(A_h,B_{\sigma(h)}).\eqno{(*)}
\]
This is an alternating sum over all $h$-tuples of paths joining $A_1,\dots,A_h$ with $B_1,\dots,B_h$ in arbitrary order, taken with the appropriate sign. Such an $h$-tuple can be either noncrossing or have crossings; in the former case, it corresponds to the identity permutation. So we only need to show that the alternating sum over all $h$-tuples of paths with crossings is zero.

To do this, we construct a weight-preserving bijection on the set of $h$-tuples of  paths with crossings. It is constructed as follows. We start with an arbitrary $h$-tuple of paths $(P_1,\dots,P_h)$ with crossings, corresponding to a permutation $\sigma$. Consider the path $P_i$ with the smallest number of its starting point, such that it meets another path, say, $P_j$. Let $C$ be the first common point of these paths. Now let us ``switch the tails'' of these two paths: let $\widetilde P_i$ be the path which coincides with $P_i$ before $C$ and with $P_j$ after $C$. Likewise, let  $\widetilde P_j$ coincide with $P_j$ before $C$ and with $P_i$ after $C$ (see Figure~\ref{fig:swaptails}). We have obtained a new $h$-tuple of paths $(P_1,\dots,\widetilde P_i,\dots,\widetilde P_j,\dots,P_h)$. 

\begin{figure}[h!]

\begin{tikzpicture}[scale=.7,nodes={font=\scriptsize}]
\draw[help lines] (0,-5) grid (6,0);
\begin{scope}[ultra thick, rounded corners]
\draw[blue] (0,-5)--(0,-3)--(1,-3)--(1,-2.025)--(3,-2.025)--(3,-1)--(4,-1)--(4,0)--(6,0);
\draw[red] (2,-5)--(2,-1.95)--(4,-1.95)--(4,-1)--(6,-1);
\draw[green] (1,-5)--(1,-4)--(3,-4)--(3,-3)--(4,-3)--(4,-2.025)--(6,-2.025);
\end{scope}
\foreach \x in {1,...,3} {
\draw[fill] (\x-1,-5) circle [radius=0.1];
\node[below right] at (\x-1.2,-5) {$A_{\x}$};
\draw[fill] (6,-\x+1) circle [radius=0.1];
\node[below right] at (6,-\x+1) {$B_{\x}$};
\draw[fill] (2,-2) circle [radius=0.1];
\node[above left] at (2,-2) {$C$};
};
\end{tikzpicture}
\qquad
\begin{tikzpicture}[scale=.7,nodes={font=\scriptsize}]
\draw[help lines] (0,-5) grid (6,0);
\begin{scope}[ultra thick, rounded corners]
\draw[blue] (0,-5)--(0,-3)--(1,-3)--(1,-2.025)--(4,-2.025)--(4,-1)--(6,-1);
\draw[red] (2,-5)--(2,-1.95)--(3,-1.95)--(3,-1)--(4,-1)--(4,0)--(6,0);
\draw[green] (1,-5)--(1,-4)--(3,-4)--(3,-3)--(4,-3)--(4,-2.025)--(6,-2.025);
\end{scope}
\foreach \x in {1,...,3} {
\draw[fill] (\x-1,-5) circle [radius=0.1];
\node[below right] at (\x-1.2,-5) {$A_{\x}$};
\draw[fill] (6,-\x+1) circle [radius=0.1];
\node[below right] at (6,-\x+1) {$B_{\x}$};
};
\draw[fill] (2,-2) circle [radius=0.1];
\node[above left] at (2,-2) {$C$};
\end{tikzpicture}
\caption{Involution on the set of $h$-tuples of crossing paths}\label{fig:swaptails}
\end{figure}
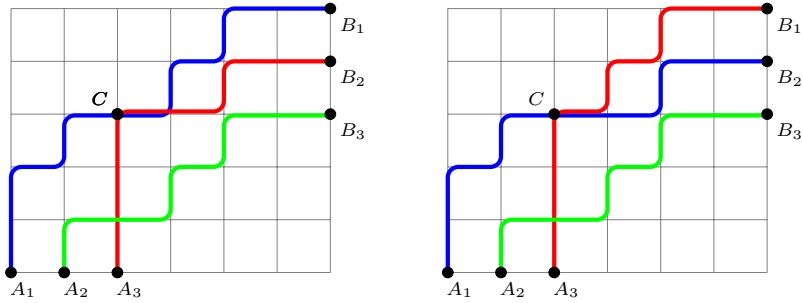

This map on the $h$-tuples of crossing paths is involutive. Obviously,
$\wt P_i\wt P_j=\wt \widetilde P_i\wt\widetilde P_j$, so it is weight-preserving. Finally, the permutations corresponding to $(P_1,\dots,P_h)$ and $(P_1,\dots,\widetilde P_i,\dots,\widetilde P_j,\dots,P_h)$ are obtained from each other by multiplication by a transposition, so they are counted in the expression $(*)$ with the opposite signs. This means that the total contribution of this pair of $h$-tuples of paths into $(*)$ is zero, and the determinant is equal to the sum of weights over all $h$-tuples of noncrossing paths. The lemma is proven.
\end{proof}

\section{Schubert polynomials}\label{sec:schubert}

\subsection{Definition of Schubert polynomials}\label{ssec:defschubert} Schubert polynomials were introduced for studying the cohomology ring of a full flag variety. Their definition is due to A.~Lascoux and M.-P.~Sch\"utzeberger \cite{LascouxSchutzenberger82}; implicitly it appeared in the paper \cite{BernsteinGelfandGelfand73} by J.~Bernstein, I.~Gelfand and S.~Gelfand. Let us recall it.

Let $S_n$ denote the symmetric group, i.e., the group of permutations of an $n$-element set $\{1,\dots, n\}$. We will use the one-line notation for permutations; i.e., $(1342)$ is a permutation taking 1 to 1, 2 to 3, 3 to 4 and 4 to 2. Let $s_i=(1,2,\dots,i+1,i,i+2,\dots,n)$ denote the $i$-th simple transposition, and $w_0=(n,n-1,\dots,2,1)$ the element of maximal length. We mean by $\ell(w)$ the \emph{length} of the permutation $w$, i.e., the minimal number $k$ of simple transpositions $s_{i_1},\dots,s_{i_k}$ such that $w=s_{i_1}\dots s_{i_k}$.

For each $i\leq n-1$, define a \emph{divided difference operator} $\partial_i\colon \ZZ[x_1,\dots,x_n]\to \ZZ[x_1,\dots,x_n]$ as follows:
\[
\partial_i f(x_1,\dots,x_n)=\frac{f(x_1,\dots,x_n)-f(x_1,\dots,x_{i+1},x_i,\dots,x_n)}{x_i-x_{i+1}}.
\]

\begin{definition} Let $w\in S_n$ be a permutation. Consider a reduced decomposition of $w_0 w$ into a product of simple transpositions:
\[
w_0w=s_{i_1}s_{i_2}\dots s_{i_k}.
\]
Then the Schubert polynomial corresponding to $w$ is defined as follows:
\[
\fS_w=\partial_{i_k}\dots\partial_{i_2}\partial_{i_1}(x_1^{n-1}x_2^{n-2}\dots x_{n-1}).
\]
(Note that the order of indices is reversed). 
\end{definition}

Since the divided difference operators satisfy the braid relations: 
\begin{eqnarray*}
\partial_i\partial_j=\partial_j\partial_i &\text{for } |i-j|>1,\\ \partial_i\partial_{i+1}\partial_i=\partial_{i+1}\partial_{i}\partial_{i+1},
\end{eqnarray*}
the polynomial $\fS_w$ depends only on $w$ and not on its decomposition into simple transpositions.

There are standard embeddings $S_n\hookrightarrow S_m$ for $n<m$; under this embedding, $S_n$ permutes only the first $n$ elements of $\{1,\dots, m\}$. Sometimes we will treat permutations as elements of the direct limit $S_\infty=\lim\limits_{\to} S_n$. It can be shown (cf.~\cite{Macdonald91}) that the Schubert polynomials are stable under these embeddings, so they can be regarded as elements of the polynomial ring $\ZZ[x_1,x_2,\dots]$ in countably many variables.

\subsection{Schubert and flagged Schur polynomials}

Recall that a permutation $w\in S_n$ is said to be \emph{vexillary} if it is (2143)-avoiding, i.e., there are no four numbers $i<j<k<\ell$ such that $w(j)<w(i)<w(\ell)<w(k)$. M.~Wachs~\cite{Wachs85} shows that the Schubert polynomials of vexillary permutations can be obtained as flagged Schur polynomials for certain flags. This is done as follows.

We define the \emph{inversion sets} of a permutation $w$ as $I_i(w)=\{j\mid j>i, w(j)<w(i)\}$. It can be shown that a permutation is vexillary iff its inversion sets form a chain (see \cite{Manivel98} for details).

Let $c_i=\# I_i$. The sequence of these cardinalities $(c_1,\dots,c_n)$ is called the \emph{Lehmer code} of $w$. Denote by $\lambda(w)$ the partition obtained by arranging $c_1,\dots, c_n$ in the decreasing order. Likewise, let $b(w)=(b_1(w)\leq\dots\leq b_m(w))$ be the sequence of integers $\min I_i(w)-1$ for all nonempty $I_i(w)$, arranged in the increasing order.

\begin{theorem}[{\cite[Theorem~2.3]{Wachs85}}]\label{thm:wachs} If $w$ is vexillary, then
\[
\fS_w=s_{\lambda(w)}(b(w)).
\]
\end{theorem}

\subsection{Dominant permutations} 
\begin{definition}
A permutation $w\in S_n$ is called \emph{dominant} if it is 132-avoiding, i.e., there is no triple $i<j<k$ such that $w(i)<w(k)<w(j)$.
\end{definition}

Clearly, if a permutation is 132-avoiding, it is also 2143-avoiding. So, all dominant permutations are vexillary.

It is well-known (cf., for instance, \cite{Manivel98}) that dominant permutations are characterized by the following property: their Schubert polynomial $\fS_w$ is a monomial in $x_1,\dots,x_{n-1}$. Moreover, the Lehmer code $\lambda(w)$ of such a permutation is a partition, and  $\fS_w=\underline{(x_1,\dots,x_{n-1})}^{\lambda(w)}$.

For a dominant permutation $w$ with its code $\lambda=(\lambda_1,\dots,\lambda_n)$, let $(i_1,\dots,i_k)$ be the numbers of the rows of $\lambda$ which are strictly shorter than the previous rows:
\[
\lambda_1=\dots=\lambda_{i_1-1}>\lambda_{i_2}=\dots=\lambda_{i_3-1}>\dots.
\]
Then the flag of $w$ is as follows:
\[
b(w)=(i_1,\dots,i_1,i_2,\dots,i_2,\dots),
\]
with $i_j-i_{j-1}$ entries equal to $i_j$. 

A flagged tableau of shape $\lambda$ with this flag is unique: all its entries in the $i$-th row are equal to $i$. So we can replace this flag by $(1,2,3,\dots)$. Hence, for a dominant permutation $w$ its Schubert polynomial is a ``$0$-flagged Schur polynomial of shape $\lambda$''.

\subsection{Shifts of permutations} Let $w\in S_n$ be a permutation, and let $h$  be a positive integer. The \emph{$h$-shifted permutation} $1^h\times w\in S_{n+h}$ is defined as
\[
1^h \times w\colon i\mapsto\begin{cases} i,& i\leq h;\\ w(i-h)+h,& i>h.\end{cases}
\]
In different terms, we add $h$ elements ``in front of the set $\{1,\dots,n\}$'' acted on by permutations, such that the action of $1^h\times w$ on these new elements is trivial, and the action on the remaining elements coincides with the action of $w$.

The following proposition is straightforward.

\begin{prop} Let $w$ be a vexillary permutation with the Lehmer  code $\lambda(w)$ and flag $b(w)=(b_1,b_2,\dots)$, and $h$ a positive integer. Then:
\begin{itemize}
\item $1^h\times w$ is vexillary;

\item $\lambda(1^h\times w)=\lambda(w);$ 

\item $b(1^h\times w)=(b_1+h,b_2+h,\dots)$.
\end{itemize}
\end{prop}

This gives us a description of permutations whose Schubert polynomials are equal to $h$-flagged Schur polynomials.

\begin{remark} Of course, for a dominant permutation $w$, the permutations $1^h\times w$ are not dominant anymore.
\end{remark}

\subsection{Extension of dominant permutations} In this subsection we define an operation on dominant permutations which corresponds to $(k,\ell)$-extension of their Young diagrams.

Let $w$ be a dominant permutation. Define a \emph{$(1,0)$-extended} permutation $\widehat w[1,0]$ as follows:
\[
\widehat w[1,0]\colon \begin{cases}
1\mapsto w(1)+1;\\
2\mapsto w(1);\\
i\mapsto w(i-1),& 2<i, w(i-1)<w(1), \\
i\mapsto w(i-1)+1, & 2<i, w(i-1)>w(1).
\end{cases}
\]
This is shown on Fig.~\ref{fig:ext1}: the first entry is ``doubled'', represented by a circle, is replaced by two entries. In this picture $w=(42135)$, $\widehat w[1,0]=(542136)$.
\begin{figure}[h!]
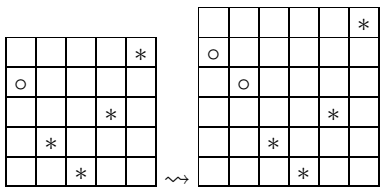

\[
\young(~~~~*,{\circ}~~~~,~~~*~,~*~~~,~~*~~) \leadsto 
\young(~~~~~*,{\circ}~~~~~,~{\circ}~~~~,~~~~*~,~~*~~~,~~~*~~)
\]
\caption{(1,0)-extension of a permutation}\label{fig:ext1}
\end{figure}

Clearly, the code of the latter permutation is obtained from the code of the former permutation exactly by the operation of $(1,0)$-extension of Young diagrams in the sense of Definition~\ref{extended2}. Moreover, $\widehat w[1,0]$ is also dominant.

In a similar manner we can define the $(0,1)$-extension of a permutation. The only difference is that the element replaced by two elements is not 1, but the preimage of $1$:
\[
\widehat w[0,1]\colon\begin{cases}
i\mapsto w(i)+1, & i\leq w^{-1}(1);\\
w^{-1}(1)+1\mapsto 1;\\
i\mapsto w(i-1)+1, & i> w^{-1}(1)+1.
\end{cases}
\]

Here is an example of the $(0,1)$-extension. In this example the original permutation is $(42135)$, and the extended permutation is $(532146)$.
\begin{figure}[h!]
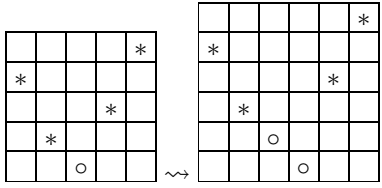

\[
\young(~~~~*,*~~~~,~~~*~,~*~~~,~~{\circ}~~) \leadsto \young(~~~~~{*},{*}~~~~~,~~~~{*}~,~{*}~~~~,~~{\circ}~~~,~~~{\circ}~~)
\]
\caption{$(0,1)$-extension of a permutation}\label{fig:ext2}
\end{figure}

A $(k,\ell)$-extension $\widehat w[k,l]$ of a dominant permutation is obtained as a result of $k$ successive applications of $(1,0)$-extension and $\ell$ applications of $(0,1)$-extension (these operations commute). Again, the Young diagram corresponding to its code is obtained from the Young diagram of $w$ by a $(k,\ell)$-extension.

\subsection{Determinantal formula for Schubert polynomials of shifted dominant permutations}

With the definitions from the two previous subsections, we can reformulate Corollary~\ref{thm:flagged2} for Schubert polynomials. The formula from that theorem will give the determinantal expression for the Schubert polynomial of an $h$-shifted dominant permutation $w$ in terms of $1$-shifted dominant permutations obtained from $w$ by extensions.

\begin{theorem}\label{thm:mainschubert}
Let $w$ be a dominant permutation, let $\lambda=(\lambda_1,\lambda_2,\dots)$ be its Lehmer code. 
 Then the Schubert polynomial of the $h$-shifted permutation $1^h\times w$ satisfies the following determinantal expression:
\begin{multline*}
\fS_{1^h\times w}(x_1,\dots,x_{h+m})=\left[(x_1^{\lambda_1+2(h-1)}x_2^{\lambda_1+2(h-2)}\dots x_{h-1}^{\lambda_1+2} x_h^{\lambda_1})
\underline{(x_{h+1},\dots, x_{h+m})}^{\lambda}\right]\times\\
\times\det\left(\frac{\fS_{1\times\widehat w[i-1,j-1]}(x_j,x_{j+1},\dots,x_{2h+m-i})}{x_j^{\lambda_1+2h+2-i-j} \underline{(x_{j+1},\dots,x_{h+m})}^{\widehat\lambda[i-1,j-1]}}\right)_{i,j=1}^h
\end{multline*}
\end{theorem}

\begin{proof} This is a reformulation of Corollary~\ref{thm:flagged2}.
\end{proof}

\section{Schubert polynomials of Richardson permutations}\label{sec:speculations} In this section we discuss some corollaries of Theorem~\ref{thm:mainschubert} and formulate several conjectures concerning Schubert polynomials.

\subsection{Catalan numbers} Fix some notation. Let $w_0(n)=(n,n-1,\dots,2,1)\in S_n$ be the element of maximal length. It is dominant, and the corresponding Young diagram is the \emph{staircase shape} $\Lambda_n=(n-1,n-2,\dots,2,1)$.

Recall that the $n$-th Catalan number is equal to the number of subdiagrams of the staircase shape $\Lambda_n$. This number has several different $q$-generalizations. Let $\Cat_q(n)$ denote the $n$-th \emph{Carlitz--Riordan $q$-Catalan number}, obtained as the weighted sum of all such subdiagrams, where a subdiagram $\mu$ has weight $q^{\frac{n(n-1)}{2}-|\mu|}$:
\[
\Cat{}_q(n)=\sum_{\mu\subset\Lambda_n} q^{\frac{n(n-1)}{2}-|\mu|}.
\]

Here is a first observation, made by Alex Woo~\cite{Woo04}.

\begin{theorem}[\cite{Woo04}] Let $w=1\times w_0(n)$. Then the principal specialization of $\fS_w$ equals the  $n$-th Carlitz--Riordan $q$-Catalan number times a power of $q$:
\[
\fS_w(1,q,q^2,q^3,\dots)=q^{\binom{n}{3}} \Cat{}_q(n).
\]
\end{theorem}

\begin{corollary}\label{cor:woo} The value of $\fS_w$ at the identity equals the $n$-th Catalan number:
\[
\fS_w(1,1,\dots)=\Cat(n).
\]
\end{corollary}

In this case $\fS_w=s_{\Lambda_n}^{(1)}$ is the 1-flagged polynomial of the staircase partition.

\subsection{Catalan--Hankel determinants} Consider the same permutation $w_0(n)$ shifted by $h$ instead of 1. Theorem~\ref{thm:mainschubert} expresses $\fS_{1^h\times w_0}$ via $\fS_{1\times \widehat{w_0(n)}[k,\ell]}$. 

Note that $\widehat{w_0(n)}[k,\ell]=w_0(n+k+\ell)$, since the extension of a staircase diagram is again a staircase diagram. This means that all 	permutations occuring in the determinantal expression are $\fS_{1\times w_0(m)}$ for $m=n,\dots,n+2h-2$.

Specializing this identity at 1 and using Corollary~\ref{cor:woo}, we get the following proposition.
\begin{prop}\label{prop:catalanhankel} The following identity holds:
\[
\fS_{1^h\times w_0(n)}(1,1,\dots)=\begin{vmatrix} \Cat(n) & \Cat(n+1) &\dots & \Cat(n+h-1)\\ \Cat(n+1) & \Cat(n+2) &\dots & \Cat(n+h)\\ \vdots & \vdots & \ddots & \vdots \\
\Cat(n+h-1) & \Cat(n+h) &\dots & \Cat(n+2h-2)\end{vmatrix}.
\]
\end{prop}

This determinant is a \emph{Hankel determinant} of the sequence of Catalan numbers; we will later refer to it as a \emph{Catalan--Hankel determinant}.

\begin{remark} Since $\fS_{1^h\times w_0(n)}$ is the flagged Schur polynomial $s_{\Lambda_n}^{(h)}$, we see from the definition of a flagged Schur polynomial that its value at 1 is equal to the number of \emph{plane partitions} with parts at most $h$ and shape $\Lambda_n$, i.e., three-dimensional Young diagrams inside a prism with base $\Lambda_n$ and height $h$. This observation appeared in~\cite[Sec.~1.2]{FominKirillov97}. 

It is also possible to establish an explicit bijection between the monomials in $\fS_{1^h\times w_0(h)}$, indexed by the so-called ``pipe dreams'' (see \cite{FominKirillov96}), and such plane partitions. See \cite{SerranoStump12} for the details.
\end{remark}

\subsection{Richardson elements in the symmetric group}

\begin{definition} Let $i_1<i_2<i_3<\dots<i_k=n$. Permutation $(i_1,i_1-1,\dots,2,1,i_2,i_2-1,\dots,i_1+1, i_3,\dots,i_2+1,\dots)\in S_n$ is called a \emph{Richardson} permutation.
\end{definition}

In different terms, a Richardson permutation is equal to the product of several permutations of type $1^h\times w_0(k)$ with nonintersecting supports. In particular, $w_0(n)$ and $1^h\times w_0(k)$ are Richardson permutations.

Richardson elements have a nice description in terms of algebraic groups: for a parabolic subgroup $P\subset \GL(n)$  consider its Levi decomposition: $P=L\ltimes U$. Then the longest element in the Weyl group $W_L$ of $L$ is a Richardson element in the Weyl group  $W\cong S_n$ of $\GL(n)$. 

Richardson permutations are not vexillary (unless they are shifted staircase permutations). However, they are obtained as products of several shifted staircase permutations with pairwise disjoint supports. So the Schubert polynomial of such a permutation is equal to the product of the Schubert polynomials of its factors.

Hence, Theorem~\ref{thm:mainschubert} shows that the Schubert polynomials of Richardson permutations can be presented as products of several determinants, and their specialization at 1 is equal to the product of several Catalan--Hankel determinants from Proposition~\ref{prop:catalanhankel}.

The study of Richardson permutations is motivated by the following question.

\begin{question} For a given $n$, find the permutations $w\in S_n$ such that $\fS_{w}(1,\dots,1)$ is as large as possible.
\end{question}

This question has a geometric interpretation: A.~Knutson and E.~Miller \cite{KnutsonMiller05} showed that the value $\fS_w(1,\dots,1)$ is equal to the degree of the \emph{matrix Schubert variety} corresponding to $w$. If $w$ satisfies the condition that, for every $(i,j)$ with $i+j>n$, either $(w_0w)^{-1}(i)\leq j$ or $w_0w(j)\leq i$, this degree equals the multiplicity of  the singular point $X_e\subset X_w$, where $X_w=\overline{B^- wB/B}\subset \GL(n)/B$ is the Schubert variety in a full flag variety corresponding to $w$. This multiplicity is a measure of ``how singular'' $X_w$ is; we are then trying to find the Schubert variety with the ``worst'' singularities (in the aforementioned sense).

These values were computed for all permutations with $n\leq 10$. It turned out that the maximum is always attained at a Richardson permutation. This allows us to formulate the following conjecture.

\begin{table}[h]
\begin{tabular}{||r|r|r||}
\hline
\hline
$n$& $w$ & $\fS_w(1,\dots,1)$\\
\hline
\hline
2 & (12) & 1\\
\hline
3 & (132) & 2\\
\hline
4 & (1432) & 5\\
\hline
5 & (15432) & 14\\
& (12543) & 14\\
& (21543) & 14\\
\hline
6 & (126543) & 84\\
& (216543) & 84\\
\hline
7 &  (1327654) &660\\
\hline
8 & (13287654) & 9438\\
\hline
9 & (132987654) & 163592\\
\hline
10 & (1,4,3,2,10,9,8,7,6,5) & 4424420\\
\hline
\hline
\end{tabular}
\caption{Permutations producing the largest values of $\fS_w(1,\dots,1)$}\label{lasttable}
\end{table}

\begin{conjecture} For a given $n$, the permutation $w\in S_n$ with the largest value of $\fS_w(1,\dots,1)$ is Richardson. 
\end{conjecture}

In Table~\ref{lasttable} we give the list of permutations $w$ corresponding to the largest value  of $\fS_w(1,\dots,1)$ and the values themselves for $n\leq 10$ (For $n=5,6$, there is more than one such permutation). This table was computed by I.~Kochulin~\cite{Kochulin12}).

However, this conjecture remains a purely experimental observation; so far we cannot see any reason why the permutation producing the largest value of its Schubert polynomial at 1 has to be Richardson.

\subsection*{Acknowledgements} E.S. was partially supported by RSCF grant 14-21-00053
within AG Laboratory NRU-HSE. 
This paper was completed during the second author's visit to the University of Warwick; he is grateful to this institution and to Professor Miles Reid for their warm hospitality.

\bibliographystyle{alpha}
\bibliography{schur}

\end{document}